\documentclass[graybox]{svmult}

\usepackage{times}
\usepackage{helvet}         
\usepackage{courier}        
\usepackage{type1cm}        
\usepackage{makeidx}         
\usepackage{graphicx}        
\usepackage{multicol}        
\usepackage[bottom]{footmisc}

\usepackage{amssymb,amsbsy,amsmath}

\makeindex             

\newcommand{\bbN}{\mathbb{N}}
\newcommand{\bbR}{\mathbb{R}}

\newcommand{\sfE}{\mathsf{E}}

\renewcommand{\d}{\mathrm{d}}

\renewcommand{\P}{\mathcal{P}}
\newcommand{\Q}{\mathcal{Q}}

\renewcommand{\S}{\Sigma}

\newcommand{\eps}{\varepsilon}
\newcommand{\s}{\sigma}
\renewcommand{\phi}{\varphi}

\newcommand{\ol}[1]{\overline{#1}}
\newcommand{\fin}{\nolinebreak\hspace{\stretch{1}}$\lhd$}
\renewcommand{\t}[1]{\widetilde{#1}}
\renewcommand{\to}{\longrightarrow}
\newcommand{\actson}{\curvearrowright}

\begin{document}

\title*{Ajtai-Szemer\'edi Theorems over quasirandom groups}
\author{Tim Austin}
\institute{Tim Austin \at Courant Institute, New York University, New York, NY 10012, United States of America\\ \email{tim@cims.nyu.edu}}
%
%
\maketitle

\abstract{Two versions of the Ajtai-Szemer\'edi Theorem are considered in the Cartesian square of a finite non-Abelian group $G$.  In case $G$ is sufficiently quasirandom, we obtain strong forms of both versions: if $E \subseteq G\times G$ is fairly dense, then $E$ contains a large number of the desired patterns for most individual choices of `common difference'.  For one of the versions, we also show that this set of good common differences is syndetic.}

\section{Introduction}
\label{sec:1}

A classical result of Ajtai and Szemer\'edi~\cite{AjtSze74} asserts the following.  For every $\delta > 0$ there is an $N_0 < \infty$ such that, if $G$ is a finite cyclic group with $|G| \geq N_0$, and $E \subseteq G \times G$ satisfies
\begin{eqnarray}\label{eq:delta-dense}
|E| \geq \delta |G|^2,
\end{eqnarray}
then there are $x,y \in G$ and $r \in G\setminus \{0\}$ such that
\begin{eqnarray}\label{eq:Ab-triples}
E \supseteq \{(x,y),(x+r,y),(x,y+r)\}.
\end{eqnarray}
A set satisfying~(\ref{eq:delta-dense}) is called \textbf{$\delta$-dense} in $G\times G$.  Subsets of $G \times G$ of the kind on the right-hand side of~(\ref{eq:Ab-triples}) are called \textbf{Abelian corners}.

Several proofs of this are now known, and it has been generalized to subsets of $G\times G$ for arbitrary finite Abelian groups $G$: see, for instance, the discussions around~\cite[Propositions 10.47 and 11.28]{TaoVu06}.

All known proofs actually give a stronger conclusion: that there is some $c > 0$ depending only on $\delta$ with the property that any $\delta$-dense set $E \subseteq G\times G$ must satisfy
\begin{eqnarray}\label{eq:counting}
\quad\quad \big|\big\{(x,y,r) \in G^3:\ \{(x,y),(x+r,y),(x,y+r)\}\subseteq E \big\}\big| \geq c|G|^3.
\end{eqnarray}

This paper considers two candidate generalizations of these results to non-Abelian groups $G$.  The most obvious non-Abelian analogs of Abelian corners are triples of the form
\begin{eqnarray}\label{eq:patt1}
\{(x,y),(gx,y),(x,gy)\}, \quad g \neq e,
\end{eqnarray}
where $e$ is the identity of $G$.  We refer to these as \textbf{na\"\i ve corners}.  A second possibility is triples of the form
\begin{eqnarray}\label{eq:patt2}
\{(x,y),(gx,y),(gx,gy)\}, \quad g \neq e.
\end{eqnarray}
Triples of this second kind first appeared in~\cite{BerMcCZha97} (in the related setting of infinite amenable groups and their probability-preserving actions), so we refer to them as \textbf{Bergelson-McCutcheon-Zhang} (\textbf{`BMZ'}) \textbf{corners}.  Solymosi~\cite{Sol13} has studied BMZ corners in finite groups in; he writes them as
\[\{(x,y),(xg,y),(x,gy)\}, \quad g \neq e,\]
but these may be identified with BMZ corners by first writing in terms of $y$, $g$ and $x' := xg$ and then applying the transformation $(a,b)\mapsto (a^{-1},b)$.

When $G$ is Abelian, na\"\i ve corners are Abelian corners, and BMZ corners are equivalent to them by a simple change of variables, but this is not true for non-Abelian $G$. As discussed in~\cite{BerMcCZha97}, when searching for BMZ corners inside a $\delta$-dense subset of $G\times G$, some methods from the Abelian setting have useful generalizations which do not seem to apply to na\"\i ve corners.

Either of these patterns, and indeed many others, must appear inside a $\delta$-dense subset of $G\times G$ for any sufficiently large finite group $G$.  This is simply because such a $G$ must contain a fairly large Abelian subgroup~\cite{ErdStr75}, to whose cosets one can apply the known results for Abelian groups.

Our first theorem states this result formally.  Here and for the rest of the paper, we also allow $G$ to be a non-finite compact metrizable group, since our methods handle these with little change.  We generally refer to these simply as `compact' groups, suppressing the assumption that they are metrizable.  If $G$ is a compact group then its Haar probability measure is denoted by $m_G$; if $H \leq G$ is a closed subgroup and $g \in G$, then $m_{gH}$ and $m_{Hg}$ denote respectively the left- and right-translates of $m_H$, regarded as a measure on $G$, by the element $g$.

\vspace{7pt}

\noindent\textbf{Theorem A. }
\emph{
For any $\delta > 0$ and any $k \in \bbN$ there is an $N_0 < \infty$ for which the following holds.  If $G$ is a compact group with $|G| \geq N_0$ (in particular, if $G$ is not finite), and $E \subseteq G^k$ is Borel and satisfies $m_{G^k}(E) \geq \delta$, then there are
\[(x_1,\ldots,x_k) \in G^k \quad \hbox{and} \quad g \in G\setminus \{e\}\]
such that
\[E \supseteq \{(x^{\eta_1}_1,\ldots,x^{\eta_k}_k):\ (\eta_1,\ldots,\eta_k) \in \{0,1\}^k\},\]
where for each $i \leq k$ we set $x_i^0 := x_i$ and $x_i^1 := gx_i$.
}

\vspace{7pt}

The analogous result in which we set $x_i^1 := x_ig$ follows immediately, simply by applying Theorem A to the set
\[\{(y_1^{-1},\ldots,y_k^{-1}):\ (y_1,\ldots,y_k) \in E\}.\]

The simple proof of Theorem A will be given in Section~\ref{sec:gen-gps}.  However, the proof does not give a generalization of the `counting' result of inequality~(\ref{eq:counting}) to non-Abelian $G$, because it counts only those patterns for which the parameter $g$ lies in some Abelian subgroup of $G$, which may be very small compared to $G$ itself.

It seems important to begin this paper with Theorem A, but our main results are of a different kind.  These assert that, if $G$ is sufficiently quasirandom (as defined in~\cite{Gow08}) in terms of $\delta$, and $E\subseteq G\times G$ is $\delta$-dense, then there is a large set of possible `common differences' $g$ such that $E$ must contain both many na\"\i ve corners and many BMZ corners with that specific choice of $g$.  In contrast to the Abelian setting, the assumption of quasirandomness allows one to obtain reasonable and explicit bounds.  The current best bounds in the Abelian case are those of Shkredov~\cite{Shk05,Shk06}.

\vspace{7pt}

\noindent\textbf{Theorem B. }
\emph{Let $\eps > 0$ and suppose that $G$ is a $D$-quasirandom compact group.  Let $E \subseteq G\times G$ be Borel and set $\delta:= m_{G\times G}(E)$.  Then the set
\[A := \big\{g \in G:\ m_{G\times G}\big\{(x,y):\ \{(x,y),(gx,y),(x,gy)\} \subseteq E\big\} \geq \delta^3 - \eps\}\]
has $m_G(A) \geq 1 - 2\sqrt{3}D^{-1/4}/\eps$.  In particular, if $D \gg (1/\eps)^4$ then $A$ is most of $G$.
}

\vspace{7pt}

This is proved in Section~\ref{sec:nai}.  As far as I know, this is the first result for na\"\i ve corners in non-Abelian groups, although the proof uses only some simple estimates and classical representation theory.

Similar methods give an analogous result for BMZ corners, although the bounds we obtain are more complicated, and it seems unlikely that they are optimal.  In formulating the results for BMZ corners, we no longer take care over the exact values of constants: universal constants are simply denoted by $\rm{O}(1)$.

\vspace{7pt}

\noindent\textbf{Theorem C. }
\emph{For any $\eps \in (0,1/2)$ there is a
\begin{eqnarray}\label{eq:D-lower-bound}
D = \exp\big((1/\eps)^{\rm{O}(1)}\big)
\end{eqnarray}
for which the following holds.  Let $G$ be a $D$-quasirandom compact group, let $E \subseteq G \times G$ be Borel, and set $\delta := m_{G\times G}(E)$.  Then the set
\[B := \big\{g \in G:\ m_{G\times G}\big\{(x,y):\ \{(x,y),(gx,y),(gx,gy)\} \subseteq E\big\} \geq \delta^4 - \eps\big\}\]
has $m_G(B) \geq 1 - \eps$.}

\vspace{7pt}

In the setting of Theorem C, a closely related argument proves another fact about the set $B$ of `good common differences'.

\vspace{7pt}

\noindent\textbf{Theorem D. }
\emph{For any $\eps \in (0,1/2)$ there is a
\begin{eqnarray}\label{eq:D-lower-bound2}
D = \exp\big(\exp\big((1/\eps)^{\rm{O}(1)}\big)\big)
\end{eqnarray}
for which the following holds.  If $G$ is a $D$-quasirandom compact group and $E \subseteq G\times G$ is Borel, then the set $B$ defined in Theorem C is $K$-left-syndetic for some
\[K = \exp\big((1/\eps)^{\rm{O}(1)}\big).\]}

\vspace{7pt}

This gives a different sense in which the set $B$ is `large'. We henceforth abbreviate `$K$-left-syndetic' to `$K$-syndetic'.

Qualitative versions of Theorems C and D have recently been proved by Bergelson, Robertson and Zorin-Kranich~\cite{BerRobZorKra14}.  Their work uses ergodic theory for actions of certain ultraproducts of sequences of finite groups, which are of course highly infinite; it gives no explicit control on $D$ or $K$ in terms of $\eps$.  Our proof below uses only elementary inequalities and representation theory, so should be simpler for the reader not versed in non-standard analysis.  In this respect, the present paper relates to~\cite{BerRobZorKra14} as did the earlier work~\cite{Aus--quantBerTao} to~\cite{BerTao14}.  We prove Theorems C and D in Section~\ref{sec:Berg}.

Theorems B, C and D together raise the following.

\begin{question}
Is it true that, if $G$ is sufficiently quasirandom in terms of $\eps \in (0,1/2)$, then the set $A$ from Theorem B is $K$-syndetic for some $K$ depending only on $\eps$? \fin
\end{question}

Intuitively, quasirandom groups are very far from Abelian groups.  Taken together, the Ajtai-Szemer\'edi Theorem for Abelian groups and our Theorems B and C make it natural to ask whether a lower bound of the kind in~(\ref{eq:counting}) holds for either na\"\i ve or BMZ corners and for all finite groups $G$ that are sufficiently large.

In fact, this is true for BMZ corners, since the original graph-theoretic proof of Ajtai and Szemer\'edi generalizes fairly easily to those: see~\cite{Sol13}.  This observation has a counterpart in ergodic theory, where related questions about multiple recurrence have been studied~\cite{BerMcCZha97,BerMcC07,ChuZorKra14}.  This is why the pattern~(\ref{eq:patt2}) was first introduced into ergodic theory by Bergelson, McCutcheon and Zhang in~\cite{BerMcCZha97}.

However, I do not see a way to apply that graph-theoretic argument in looking for na\"\i ve corners, so for those the following question may be of interest.

\begin{question}
Is it true that for every $\delta > 0$ there is a $c > 0$ such that, whenever $G$ is a compact group and $E \subseteq G\times G$ satisfies $m_{G^2}(E) \geq \delta$, one also has
\[m_{G^3}\big\{(x,y,g) \in G^3:\ \{(x,y),(gx,y),(x,gy)\} \subseteq E\big\} \geq c?\]
\fin
\end{question}

%
%
%
%

\section{Proof of Theorem A}\label{sec:gen-gps}

\begin{proof}[Theorem A]\smartqed
\quad\emph{Step 1: finite groups.}\quad The theorem is well-known among finite Abelian groups: see, for instance,~\cite[Proposition 11.28]{TaoVu06}.  Fix $\delta > 0$ and $k \in \bbN$, and let $n$ be so large that, if $A$ is any finite Abelian group of order at least $n$ and $F\subseteq A^k$ satisfies $m_{A^k}(F) \geq \delta$, then $F$ must contain a pattern of the kind in Theorem A.

By a classical result of Erd\H{o}s and Straus~\cite{ErdStr75}, we may now choose $N$ so large that any finite group $G$ of order at least $N$ contains an Abelian subgroup $H$ of order at least $n$ (see also~\cite{Pyb97} for an essentially optimal estimate of $N$ in terms of $n$).  Suppose that $E \subseteq G^k$ has $m_{G^k}(E) \geq \delta$.  Then
\[m_{G^k}(E) = \int\cdots\int m_{Hg_1\times \cdots \times Hg_k}(E)\,m_{H\backslash G}(\d (Hg_k))\cdots m_{H\backslash G}(\d(Hg_1)),\]
so there are some cosets $Hg_1$, \ldots, $Hg_k$ for which
\[m_{Hg_1\times Hg_2\times \cdots \times Hg_k}(E) \geq \delta.\]
Define $F \subseteq H^k$ by
\[E\cap (Hg_1\times Hg_2\times \cdots \times Hg_k) = F \cdot (g_1,\ldots,g_k).\]

Applying the Abelian case of the theorem to $F$ gives some
\[(x_1,\ldots,x_k) \in H^k \quad \hbox{and} \quad h \in H\setminus \{e\}\]
such that
\[F \supseteq \{(x_1^{\eta_1},\ldots,x_k^{\eta_k}):\ (\eta_1,\ldots,\eta_k) \in \{0,1\}^k\},\]
with the obvious analog of the notation in the statement of Theorem A.  Translating back to $E$, this gives
\[E \supseteq \{((x_1g_1)^{\eta_1},\ldots,(x_kg_k)^{\eta_k}):\ (\eta_1,\ldots,\eta_k) \in \{0,1\}^k\}.\]

\vspace{7pt}

\emph{Step 2: other compact Lie groups.}\quad If $G$ is a non-finite compact Lie group, then it contains a nontrivial toral subgroup.  This, in turn, contains finite cyclic subgroups of arbitrarily large cardinality.  Letting $H$ such a subgroup of cardinality at least $n$, we may complete the proof as in Step 1.

\vspace{7pt}

\emph{Step 3: general compact groups.}\quad Finally, let $G$ be an arbitrary compact group which is not finite or a Lie group.  As a standard consequence of the Peter-Weyl Theorem~\cite[Section III.3]{Brotom85}, there are a continuous surjective homomorphism $\pi:G\to \ol{G}$ to a compact Lie group such that the set
\begin{eqnarray}\label{eq:define-Ebar}
\ol{E} := \{(\ol{x}_1,\ldots,\ol{x}_k) \in \ol{G}^k:\ m_{\pi^{-1}\{\ol{x}_1\}\times \cdots \times \pi^{-1}\{\ol{x}_k\}}(E) > 1 - 2^{-k-1}\}
\end{eqnarray}
satisfies
\[m_{G^k}(E\triangle (\pi\times \cdots \times \pi)^{-1}(\ol{E})) < \delta/2,\]
and hence $m_{\ol{G}^k}(\ol{E}) > \delta/2$.

Since $G$ is not finite or a Lie group, we may choose $\ol{G}$ to have cardinality as large as we please (allowing infinity if necessary).  Having done so, either Step 1 or Step 2 gives some
\[(\ol{x}_1,\ldots,\ol{x}_k) \in \ol{G}^k \quad \hbox{and} \quad \ol{g} \in \ol{G}\setminus \{\ol{e}\}\]
such that
\[\ol{E} \supseteq \{(\ol{x}_1^{\eta_1},\ldots,\ol{x}_k^{\eta_k}):\ (\eta_1,\ldots,\eta_k) \in \{0,1\}^k\},\]
with the obvious analog of the notation in the statement of Theorem A.

Finally, consider lifts $x_1$, \ldots, $x_k$ and $g$ chosen independently at random from the Haar measures $m_{\pi^{-1}\{\ol{x}_1\}}$, \ldots, $m_{\pi^{-1}\{\ol{x}_k\}}$ and $m_{\pi^{-1}\{\ol{g}\}}$.  Observe that each $gx_i$ is then a random lift of $\ol{g}\ol{x}_i$ with distribution $m_{\pi^{-1}\{\ol{g}\ol{x}_i\}}$.  Define each $x_i^{\eta_i}$ using these lifts in the usual way.  Then it follows from the definition~(\ref{eq:define-Ebar}) that each of the events
\[(x_1^{\eta_1},\ldots,x_k^{\eta_k}) \in E \quad \hbox{for}\ (\eta_1,\ldots,\eta_k) \in \{0,1\}^k\]
has probability at least $1 - 2^{-k-1}$.  Therefore, by the first-moment bound, there is some choice of $x_1$, \ldots, $x_k$ and $g$ for which all of these events occur simultaneously.  This completes the proof.
\qed
\end{proof}

\section{Preliminary discussion of the results for quasirandom groups}\label{sec:prelim}

The proofs of Theorems B, C and D have several common elements.  This section introduces some of those.

It will be convenient to use some simple notation and terminology from ergodic theory.  Given a compact group $G$, a \textbf{probability $(G\times G)$-space} is a tuple $(Y,\S,\nu,S,T)$ in which $(Y,\S,\nu)$ is a probability space and $S$ and $T$ are two commuting, $\nu$-preserving $G$-actions on that space.  Since they commute, $S$ and $T$ together define an action of $G\times G$, hence the name.  Given such a probability $(G\times G)$-space, we will write $\S^S$, $\S^T$ and $\S^{S,T}$ for the $\s$-subalgebras of sets in $\S$ that are invariant under $S$, $T$, or the whole $(G\times G)$-action, respectively.

For example, let $X := G\times G$ with the measure $\mu = m_X$, let $\S_X$ be the Borel $\s$-algebra of $X$, and define
\[S^g(x,y) := (gx,y) \quad \hbox{and} \quad T^g(x,y) := (x,gy).\]
This turns $(X,\S_X,\mu,S,T)$ into a probability $(G\times G)$-space.  It will appear repeatedly below.

Now consider measurable functions $f_1,f_2,f_3:X\to [-1,1]$, and suppose $G$ is highly quasirandom.  For Theorems B, C and D we need to estimate the values taken by
\[\hbox{either} \quad \int f_1\cdot f_2S^g \cdot f_3T^g\,\d\mu \quad \hbox{or} \quad \int f_1\cdot f_2S^g \cdot f_3S^gT^g\,\d\mu\]
for `typical' group elements $g$.  We ultimately need to do this in case $f_1 = f_2 = f_3 = 1_E$, but allowing more flexibility will be important for the proofs.  The strategy for these estimates takes a form that has become well-known in additive combinatorics: each of the functions $f_i$ will be decomposed as
\[f_i = f_i^\circ + f_i^\perp\]
in such a way that the functions $f_i^\perp$ contribute very little to the integrals of interest for `most' group elements $g$, while the functions $f_i^\circ$ have some extra `structure' which makes the estimate of that integral easier.  In such a decomposition, the negligible terms $f_i^\perp$ are often called `quasirandom'.  Here, `most' group elements will mean those lying in a very large subset of $G$ in the case of Theorem B or C, or those lying in a suitable syndetic subset in the case of Theorem D.  A nice discussion of this methodology can be found in~\cite{Gow10}.

Different decompositions are required for studying na\"\i ve and BMZ corners.  For the former, the required decompositions are fairly simple, and will be introduced within the proof of Theorem B.  For BMZ corners we need a rather more complicated construction, based on the regularity lemma of Frieze and Kannan~\cite{FriKan99}; this will be explained separately in Subsection~\ref{subs:struct-choice}.

\subsection{Estimates from probability}

The following simple lemma plays the r\^ole of the classical van der Corput estimate in the present paper.  It will be the basis for several other estimates later.

\begin{lemma}\label{lem:vdC}
Let $(Y,\S,\nu)$ be a probability space, let $V$ be a real or complex Hilbert space with inner product $\langle \cdot,\cdot\rangle$ and corresponding norm $\|\cdot\|$, and let $y \mapsto u_y$ be a strongly measurable function $Y \to V$.  Let $v$ be a unit vector in $V$.  Then
\[\int |\langle v,u_y\rangle |\,\nu(\d y) \leq \sqrt{\iint |\langle u_y,u_{y'}\rangle|\,\nu(\d y)\,\nu(\d y')}.\]
\end{lemma}

\begin{proof}\smartqed
Define
\[\phi(y) := \ol{\langle v,u_y \rangle}/|\langle v,u_y\rangle|,\]
using the convention that $\phi(y) = 1$ if $\langle v,u_y\rangle = 0$. So $\phi$ takes values in the unit circle, and more specifically in $\{-1,1\}$ if $V$ is real.  This is a measurable function of $y$, and
\begin{multline*}
\int |\langle v,u_y\rangle |\,\nu(\d y) = \int \phi(y)\langle v,u_y\rangle \,\nu(\d y) = \Big\langle v,\int \phi(y)u_y\,\nu(\d y)\Big\rangle\\
\leq \Big\|\int \phi(y)u_y\,\nu(\d y)\Big\| = \sqrt{\iint \langle \phi(y)u_y,\phi(y')u_{y'}\rangle\,\nu(\d y)\,\nu(\d y')}\\
\leq \sqrt{\iint |\langle u_y,u_{y'}\rangle|\,\nu(\d y)\,\nu(\d y')}.
\end{multline*}
\qed\end{proof}

We will also need the following very general probabilistic estimate.  It can be found as~\cite[Lemma 1.6]{Chu09}, and then~\cite{BerRobZorKra14} cites it for a similar purpose to ours.

\begin{lemma}\label{lem:Chu}
\smartqed Let $(Y,\S,\nu)$ be a probability space, let $\S_1$, \ldots, $\S_k$ be $\s$-subalgebras of $\S$, and let $f$ be a bounded non-negative measurable function on $Y$.  Then
\[\int f\cdot \prod_{i=1}^k\sfE(f\,|\,\S_i)\,\d\nu \geq \Big(\int f\,\d\nu\Big)^{k+1}.\]
\qed
\end{lemma}

\subsection{Estimates from representation theory}

All of the representation theory in this paper concerns either unitary representations on complex Hilbert spaces or orthogonal representations on real Hilbert spaces.

The results we need from representation theory are all standard, and can be found in many textbooks.  A good reference for our purposes is~\cite[Chapters II and III]{Brotom85}; that book actually focuses on compact Lie groups, but all the facts we cite from it clearly hold for arbitrary compact groups.

If $\pi:G\actson V$ is a unitary or orthogonal representation, then $P^\pi$ denotes the orthogonal projection from $V$ onto the subspace of $\pi(G)$-fixed vectors.

A unitary (resp. orthogonal) representation is \textbf{$D$-quasirandom} if it has no irreducible subrepresentations of complex (resp. real) dimension less than $D$.  Following~\cite{Gow08}, the group $G$ itself is $D$-quasirandom if and only if all its non-identity irreducible unitary representations are $D$-quasirandom.  Quasirandomness will be exploited in this paper by way of the following two lemmas.

\begin{lemma}\label{lem:quasirand-tensor}
If $G$ is a $D$-quasirandom compact group, $\pi:G\actson V$ is a unitary representation, and $u,v \in V$, then
\[\big\|P^{\pi\otimes \pi}(u\otimes v) - P^\pi u \otimes P^\pi v\big\|_{V\otimes V} \leq D^{-1/2}\|u\|_V\|v\|_V,\]
where $\|\cdot\|_V$ denotes the norm on $V$ and $\|\cdot\|_{V\otimes V}$ denotes the Hilbertian tensor product of that norm on $V\otimes V$.
\end{lemma}

\begin{proof}\smartqed
This is a routine consequence of Schur's Lemma: it can be found as~\cite[Lemma 2]{Aus--quantBerTao}.
\qed\end{proof}

Unlike the above, the next lemma will be needed for both unitary and orthogonal representations.

\begin{lemma}\label{lem:quasirand-mixing}
If $G$ is a $D$-quasirandom compact group, $\pi:G\actson V$ is a unitary or orthogonal representation, and $u,v \in V$, then
\[\int |\langle u,\pi^g v\rangle_V - \langle P^\pi u,P^\pi v\rangle_V|^2\,\D g \leq D^{-1}\|u\|_V^2\|v\|_V^2.\]
\end{lemma}

\begin{proof}\smartqed
In the unitary case, this is~\cite[Corollary 3]{Aus--quantBerTao} (although the proof given there contained an error, corrected in~\cite{Aus--quantBerTao--erratum}).

Now suppose that $\pi$ is an orthogonal representation.  Then its complexification $\pi^{\mathbb{C}}:G\actson V\otimes_\bbR \mathbb{C}$ is a unitary representation, which may also be regarded as an orthogonal representation isomorphic to $\pi\oplus \pi$ (see, for instance,~\cite[Section II.6]{Brotom85}).  A simple calculation shows that $P^{\pi^{\mathbb{C}}}$ is simply the complexification of $P^\pi$.  Hence the desired inequality for $\pi$ follows from its counterpart for $\pi^{\mathbb{C}}$.
\qed\end{proof}

\section{Na\"\i ve corners}\label{sec:nai}

This section proves Theorem B.  The key to this proof is the following proposition, which roughly asserts that `correlations' across na\"\i ve corners almost vanish unless one starts with functions that have some nontrivial structure.

Let $(X,\S_X,\mu,S,T)$ be as in Section~\ref{sec:prelim}.

\begin{proposition}\label{prop:kill-random-f3}
If $f_1,f_2,f_3:X\to [-1,1]$, and either $\sfE(f_2\,|\,\S^S_X) = 0$ or $\sfE(f_3\,|\,\S^T_X) = 0$, then
\[\int\Big|\int f_1\cdot f_2S^g \cdot f_3T^g\,\d\mu\Big|\,\d g \leq \sqrt{3}D^{-1/4}.\]
\end{proposition}

Proposition~\ref{prop:kill-random-f3} will in turn be deduced from a kind of `mixing' estimate, formulated in the next lemma.  Before stating the lemma, it will be helpful to have some notation for representations of product groups.  In this section the representations will all be unitary, even though the functions in Proposition~\ref{prop:kill-random-f3} are real-valued, since we will need to make another direct appeal to the Schur Orthogonality Relations.  Given any two unitary representation $\pi:G\actson V$ and $\theta:G\actson W$, one obtains a representation $G\times G\actson V\otimes W$ by defining
\[(\pi\boxtimes \theta)^{(g,h)} := \pi^g\otimes \theta^h.\]
It is a standard result that $\pi\boxtimes \theta$ is an irreducible $(G\times G)$-representation if each of $\pi$ and $\theta$ is irreducible, and that all irreducibles of $G\times G$ arise this way~\cite[Proposition II.4.14]{Brotom85}.  (On the other hand, as far as I know there is no standard notation for $\pi \boxtimes \theta$.)  It follows at once that if $G$ is $D$-quasirandom, then so is $G\times G$.

\begin{lemma}
\label{lem:quasirand-mixing2}
Let $G$ be $D$-quasirandom, let $\rho:G\times G\actson U$ be a unitary representation, and let $u \in U$.  Then
\[\Big\|\int \rho^{(g,g^{-1})}u\,\d g - P^{\rho}u\Big\| \leq D^{-1}\|u\|.\]
\end{lemma}

\begin{proof}\smartqed
By decomposing $\rho$ into irreducibles, it suffices to prove this when $\rho$ is itself a nontrivial irreducible.  According to the discussion above, this means that $\rho = \pi\boxtimes \theta$ for some irreducible unitary representations $\pi:G\actson V$ and $\theta:G\actson W$, not both the identity.

In this case we will prove that
\begin{eqnarray}\label{eq:swapop}
\int \rho^{(g,g^{-1})}(v\otimes w)\,\d g = \int (\pi^gv\otimes \theta^{g^{-1}}w)\,\d g = \left\{\begin{array}{ll}
\frac{1}{\dim(V)}(w\otimes v) & \ \  \hbox{if}\ \pi = \theta\\ 0& \ \  \hbox{else}\end{array}\right.
\end{eqnarray}
for any $v \in V$ and $w \in W$.  Since $\pi$ and $\theta$ cannot both be the identity, if they are equal then $\dim(V) \geq D$, by the assumption of quasirandomness. Therefore~(\ref{eq:swapop}) implies that
\[\Big\|\int \rho^{(g,g^{-1})}u\,\d g\Big\| \leq D^{-1}\|u\| \quad \forall u \in U.\]

To prove~(\ref{eq:swapop}), it suffices to check the inner products of the two sides of this equation against another element of $U$, which can also be of tensor product form.  Thus the desired equality becomes
\begin{eqnarray*}
\int \big\langle \pi^gv\otimes \theta^{g^{-1}}w,v'\otimes w'\big\rangle\,\d g &=& \int \langle \pi^gv,v'\rangle \langle \theta^{g^{-1}}w,w'\rangle\,\d g \\
&=& \left\{\begin{array}{ll}
\frac{1}{\dim(V)}\langle w,v'\rangle\langle v,w'\rangle & \quad \hbox{if}\ \pi = \theta,\\ 0& \quad \hbox{else}.\end{array}\right.
\end{eqnarray*}
This is now one of the standard Schur Orthogonality Relations: see~\cite[Theorem II.4.5(ii)]{Brotom85}.
\qed\end{proof}

\begin{corollary}
\label{cor:quasirand-mixing2}
If $G$ is $D$-quasirandom and $F_1,F_2 \in L_{\mathbb{C}}^2(G\times G)$, then
\[\int \Big|\int F_1 \cdot \ol{F_2}T^gS^{g^{-1}}\,\d\mu - \int F_1\,\d\mu\int \ol{F_2}\,\d\mu\Big|\,\d g \leq 2D^{-1/2}\|F_1\|_2\|F_2\|_2.\]
\end{corollary}

\begin{remark}
The use of $\ol{F_2}$ rather than $F_2$ on the left-hand side here is only for the sake of convenience.  With this choice, the integral  $\int F_1 \cdot \ol{F_2}T^gS^{g^{-1}}\,\d\mu$ is the Hermitian inner product in $L_{\mathbb{C}}^2(G\times G)$, which leads more easily to an application of Lemma~\ref{lem:quasirand-mixing2}. \fin
\end{remark}

\begin{proof}\smartqed
By replacing each $F_i$ with $F_i - \int F_i\,\d\mu$, we may assume that they are both orthogonal to the constant functions.  In this case we will show that
\[\int \Big|\int F_1\cdot\ol{F_2}T^gS^{g^{-1}}\,\d \mu\Big|^2\,\d g \leq 2D^{-1}\|F_1\|^2\|F_2\|^2,\]
from which the result follows by the Cauchy--Bunyakowksi--Schwartz inequality.

This follows using a standard tensor-product trick.  Let $\t{X} := X\times X$, $\t{\mu} := \mu\otimes \mu$, $\t{S} := S\times S$ and $\t{T} := T\times T$.  Then the integral above is equal to
\[\int_G \int_{\t{X}} \t{F}_1\cdot\ol{\t{F}_2}\t{T}^g\t{S}^{g^{-1}}\,\d \t{\mu}\,\d g = \int_{\t{X}} \t{F}_1 \Big(\int\ol{\t{F}_2}\t{T}^g\t{S}^{g^{-1}}\,\d g\Big)\,\d \t{\mu},\]
where $\t{F}_i(x,y,x',y') := F_i(x,y)\ol{F_i(x',y')}$. Applying Lemma~\ref{lem:quasirand-mixing2} to the inner integral on the right here, this is at most
\begin{multline*}
\int \sfE(\t{F}_1\,|\,\S_{\t{X}}^{\t{S},\t{T}})\ol{\sfE(\t{F}_2\,|\,\S_{\t{X}}^{\t{S},\t{T}})}\,\d\t{\mu} + D^{-1}\|\t{F}_1\|_{L_{\mathbb{C}}^2(\t{\mu})}\|\t{F}_2\|_{L_{\mathbb{C}}^2(\t{\mu})}\\
\leq \big\|\sfE(\t{F}_1\,|\,\S_{\t{X}}^{\t{S},\t{T}})\big\|_{L_{\mathbb{C}}^2(\t{\mu})}\big\|\sfE(\t{F}_2\,|\,\S_{\t{X}}^{\t{S},\t{T}})\big\|_{L_{\mathbb{C}}^2(\t{\mu})} + D^{-1}\|F_1\|^2\|F_2\|^2.
\end{multline*}

Finally, in view of the product form of $\t{F}_i$ and the fact that $G\times G$ is still $D$-quasirandom, an appeal to Lemma~\ref{lem:quasirand-tensor} gives
\[\big\|\sfE(\t{F}_i\,|\,\S_{\t{X}}^{\t{S},\t{T}})\big\|_{L_{\mathbb{C}}^2(\t{\mu})} \leq D^{-1/2}\|F_i\|^2 \quad \hbox{for}\ i=1,2.\]
Substituting this into the bound above completes the proof.
\end{proof}

\begin{proof}[Proposition~\ref{prop:kill-random-f3}]\smartqed
We give the proof in case $\sfE(f_2\,|\,\S_X^S) = 0$, the other case being analogous.

For each $g \in G$, let
\[u_g := f_2S^g \cdot f_3T^g.\]
By Lemma~\ref{lem:vdC}, it suffices to prove that
\[\iint |\langle u_g,u_{g'}\rangle| \,\d g\,\d g' = \iint |\langle u_g,u_{hg}\rangle| \,\d g\,\d h \leq 3D^{-1/2}.\]
The second integrand here may be re-written as
\begin{eqnarray*}
|\langle u_g,u_{hg}\rangle| &=& \Big|\int f_2S^g\cdot f_2S^{hg}\cdot f_3T^g\cdot f_3T^{hg}\,\d\mu\Big| \\
&=& \Big|\int f_2\cdot f_2S^h\cdot f_3T^gS^{g^{-1}}\cdot f_3T^{hg}S^{g^{-1}}\,\d\mu\Big| \\
&=& \Big|\int F_{2,h}\cdot F_{3,h}T^gS^{g^{-1}}\,\d\mu\Big|,
\end{eqnarray*}
where $F_{2,h} := f_2\cdot f_2S^h$ and $F_{3,h} := f_3\cdot f_3T^h$.

Both $F_{2,h}$ and $F_{3,h}$ are real-valued and bounded by $1$ in absolute value, so Corollary~\ref{cor:quasirand-mixing2} applies to give
\[\int |\langle u_g,u_{hg}\rangle|\,\d g \leq \Big|\int f_2\cdot f_2S^h\,\d\mu \Big|\Big|\int f_3\cdot f_3T^h\,\d\mu\Big| + 2D^{-1/2}\]
for any $h$.  Now integrating over $h$, and using again that $|f_3\cdot f_3T^h| \leq 1$, the right-hand side above turns into the bound
\[\int \Big|\int f_2\cdot f_2S^h\,\d\mu \Big|\,\d h + 2D^{-1/2} \leq \Big(\int \Big(\int f_2\cdot f_2S^h\,\d\mu \Big)^2\,\d h\Big)^{-1/2} + 2D^{-1/2}.\]
By Lemma~\ref{lem:quasirand-mixing} and the fact that $\sfE(f_2\,|\,\S^S_X) = 0$, this is at most $3D^{-1/2}$.  This completes the proof.
\qed\end{proof}

\begin{proof}[Theorem B]\smartqed
We will prove that, for any function $f:G\times G \to [0,1]$, the set
\[A := \Big\{g \in G:\ \iint f(x,y)f(gx,y)f(x,gy)\,\d x\,\d y \geq \Big(\int f\,\d \mu\Big)^3 - \eps\Big\}\]
has measure at least $1 - 2\sqrt{3}D^{-1/4}/\eps$.  Applying this to $f = 1_E$ gives Theorem B.

Define new $[-1,1]$-valued functions $f_2^\perp$ and $f_3^\perp$ by the decompositions
\[f = \sfE(f\,|\,\S^S_X) + f_2^\perp = \sfE(f\,|\,\S^T_X) + f_3^\perp.\]
In the present setting, these are the decompositions of $f_2 = f$ and $f_3 = f$ into `structured' and `quasirandom' parts, as promised at the beginning of Section~\ref{sec:prelim}.  It turns out that no related decomposition $f_1^\circ + f_1^\perp$ is needed here.

Substituting the first decomposition into the middle position of the relevant integral, and then the second decomposition into the last position, we obtain
\begin{eqnarray}\label{eq:decomp-for-naive}
&&\iint f(x,y)f(gx,y)f(x,gy)\,\d x\,\d y \nonumber\\
&&= \iint f(x,y)\sfE(f\,|\,\S^S_X)(gx,y)\sfE(f\,|\,\S^T_X)(x,gy)\,\d x\,\d y \nonumber\\
&& \qquad + \iint f(x,y)f_2^\perp(gx,y)f(x,gy)\,\d x\,\d y \nonumber\\
&& \qquad + \iint f(x,y)\sfE(f\,|\,\S^S_X)(gx,y)f_3^\perp(x,gy)\,\d x\,\d y.
\end{eqnarray}

Since $\sfE(f\,|\,\S^S_X)$ is $S$-invariant and $\sfE(f\,|\,\S^T_X)$ is $T$-invariant, the first integral on the right-hand side of~(\ref{eq:decomp-for-naive}) is equal to
\[\iint f(x,y)\sfE(f\,|\,\S^S_X)(x,y)\sfE(f\,|\,\S^T_X)(x,y)\,\d x\,\d y\]
for any $g$, and this is bounded below by $\big(\int f\,\d\mu\big)^3$ by Lemma~\ref{lem:Chu}.  Re-arranging~(\ref{eq:decomp-for-naive}), it follows that
\begin{multline*}
A \supseteq A' := \Big\{g \in G:\ \Big|\iint f(x,y)f_2^\perp(gx,y)f(x,gy)\,\d x\,\d y\big|\\ 
 + \Big|\iint f(x,y)\sfE(f\,|\,\S^S_X)(gx,y)f_3^\perp(x,gy)\,\d x\,\d y\Big| \leq \eps\Big\}.
\end{multline*}

On the other hand, since $\sfE(f_2^\perp\,|\,\S^S_X) = \sfE(f_3^\perp\,|\,\S^T_X) = 0$, two appeals to Proposition~\ref{prop:kill-random-f3} give
\begin{multline*}
\int\Big|\iint f(x,y)f_2^\perp(gx,y)f(x,gy)\,\d x\,\d y\Big|\,\d g\\ + \int\Big|\iint f(x,y)\sfE(f\,|\,\S^S_X)(gx,y)f_3^\perp(x,gy)\,\d x\,\d y\Big|\,\d g \leq 2\sqrt{3}D^{-1/4}.
\end{multline*}
Hence $m_G(A') \geq 1 - 2\sqrt{3}D^{-1/4}/\eps$, by Chebyshev's Inequality, so the proof is complete.
\qed\end{proof}

\section{BMZ corners}\label{sec:Berg}

This section proves Theorems C and D.  The two proofs have much in common.  We will explain the overarching structure of both proofs first, and then separate their finer details.

In this section it now makes more sense to work with orthogonal (real) representations than with unitary ones, since we will not need any representation theoretic results beyond Lemma~\ref{lem:quasirand-mixing}.

Again let $(X,\S_X,\mu,S,T)$ be as in Section~\ref{sec:prelim}.

\subsection{Decomposition into structured and quasirandom functions}\label{subs:struct-choice}

We now describe the decompositions into `structured' and `quasirandom' functions appropriate to the analysis of the  family of triple forms
\begin{eqnarray}\label{eq:BMZ-corr}
\int f_1\cdot f_2S^g \cdot f_3S^gT^g\,\d\mu, \quad g \in G.
\end{eqnarray}
The functions $f_1$, $f_2$ and $f_3$ play different r\^oles here, so each will need its own notion of `quasirandom' and `structured' summands, even though our ultimate interest is in the case $f_1 = f_2 = f_3 = 1_E$.

As is common in this area, the appropriate notions of `quasirandomness' are measured as smallness in certain norms.  We introduce these next.

For a bounded function $f:X\to \bbR$, define
\[\|f\|_{\check{\otimes}_{1,2}} := \sup\Big\{\int f(x,y)g(x)h(y)\,\d\mu:\ \|g\|_\infty,\|h\|_\infty \leq 1\Big\}.\]
This is a classical construction in Banach space theory: the injective tensor norm ${\|\cdot\|_{L^1(m_G)\check{\otimes} L^1(m_G)}}$.  It originates in Schatten's work on tensor products of Banach spaces, where it is referred to as the `bound cross-norm'. See~\cite[Section II.3]{Scha50}, or~\cite[Chapter 3]{Rya02} for a more modern treatment.  It has now become popular in extremal combinatorics, where it is called the `box norm'.

Similarly, define
\[\|f\|_{\check{\otimes}_{1,12}} := \sup\Big\{\int f(x,y)g(y)h(x^{-1}y)\,\d\mu:\ \|g\|_\infty,\|h\|_\infty \leq 1\Big\}\]
and
\[\|f\|_{\check{\otimes}_{12,2}} := \sup\Big\{\int f(x,y)g(x^{-1}y)h(x)\,\d\mu:\ \|g\|_\infty,\|h\|_\infty \leq 1\Big\}.\]
These may also be viewed as injective tensor norms, once one chooses appropriate `coordinate axes' on $G\times G$.

For general functions $f_1$, $f_2$ and $f_3$ on $G\times G$, we will make use of Frieze and Kannan's weak version of the Szemer\'edi Regularity Lemma~\cite{FriKan99} (see also~\cite[Proposition 2.11]{Gow10} for a formulation closer to the present paper).  The next proposition gives the specific instance of this result that we need.  Recall that if $\P$ is a partition of a set $S$ and $s \in S$, then $\P(s)$ denotes the cell of $\P$ which contains $s$.

\begin{proposition}[Weak Regularity Lemma]\label{prop:reglem}\smartqed
Given a measurable function $f:X\to [-1,1]$, and also $\eta > 0$, there are partitions $\P_{1,1}$ and $\P_{1,12}$ of $G$, each into at most $\exp((1/\eta)^{\rm{O}(1)})$ cells, for which the following holds.

Define a new partition of $G\times G$ cell-wise by setting
\begin{eqnarray}\label{eq:Q1}
\Q_1(x,y):= \P_{1,1}(y)\cap \P_{1,12}(x^{-1}y),
\end{eqnarray}
and let $\sfE_1:L^\infty(\mu)\to L^\infty(\mu)$ be the operator of conditional expectation onto $\Q_1$.  Then
\[\|f - \sfE_1f\|_{\check{\otimes}_{1,12}} \leq \eta.\]
\qed
\end{proposition}

This is not the formulation of the Frieze-Kannan Regularity Lemma given in~\cite{FriKan99} or~\cite{Gow10}, but it is an easy consequence.  The methods of~\cite{FriKan99} or~\cite{Gow10} give instead a function
\[f^\circ(x,y) =  \sum_{m=1}^{m_0} \lambda_m h_m'(y)h_m''(x^{-1}y)\]
which approximates $f$ in the sense that
\[\|f - f^\circ\|_{\check{\otimes}_{1,12}} \leq \eta,\]
and where
\begin{itemize}
 \item $m_0 = (1/\eta)^{\rm{O}(1)}$,
\item each $\lambda_m \in [-1,1]$,
\item and each $h_m'$ and $h_m''$ is an indicator function on $G$.
\end{itemize}
For our purposes it is important to approximate $f$ by a function that still takes values in $[0,1]$, hence our preference for approximating by a conditional expectation of $f$ itself.  To obtain suitable partitions $\P_{1,1}$ and $\P_{1,12}$ from the function $f^\circ$ above, one simply lets $\P_{1,1}$ be generated by the level sets of all the functions $h_m'$, $m=1,2,\ldots,m_0$, and similarly for $\P_{1,12}$ using $h_m''$.

In recent years, techniques for decomposing a function into structured and quasirandom parts have become quite sophisticated: see~\cite{Gow10}, for example.  A more careful argument than ours might enable one to be more efficient in the above proposition, and perhaps ultimately improve the bound in Theorem C to $D = (1/\eps)^{\mathrm{O}(1)}$.  However, it is not the purpose of the present paper to explore this kind of enhancement, and we content ourselves with the partitions obtained above.

We will also need two variants of Proposition~\ref{prop:reglem} that are obtained from it by simple changes of variables in $G\times G$.  Given $f$ and $\eta$ as above, there are also partitions $\P_{2,1}$, $\P_{2,2}$, $\P_{3,2}$ and $\P_{3,12}$, all into at most $\exp((1/\eta)^{\rm{O}(1)})$ cells, for which the following hold: defining two new partitions of $G\times G$ by
\begin{eqnarray}
\Q_2(x,y) &:=& \P_{2,2}(x)\cap \P_{2,1}(y) \label{eq:Q2}\\ \quad \hbox{and} \quad \Q_3(x,y) &:=& \P_{3,12}(x^{-1}y)\cap \P_{3,2}(x)\label{eq:Q3},
\end{eqnarray}
and letting $\sfE_i:L^\infty(\mu)\to L^\infty(\mu)$ be the operator of conditional expectation onto $\Q_i$ for $i=2,3$, one has
\[\|f - \sfE_2f\|_{\check{\otimes}_{1,2}} \leq \eta \quad \hbox{and} \quad \|f - \sfE_3f\|_{\check{\otimes}_{12,2}} \leq \eta.\]

At some points below, we will need to write out the functions $\sfE_if$, $i=1,2,3$ in terms of more elementary summands.  Using the individual cells of $\Q_i$, one may always express
\begin{eqnarray}
(\sfE_1f)(x,y) &=& \sum_{m_1=1}^{M_1}h_{1,1,m_1}(y)h_{1,12,m_1}(x^{-1}y),\label{eq:E1f1}\\
(\sfE_2f)(x,y) &=& \sum_{m_2=1}^{M_2}h_{2,1,m_2}(y)h_{2,2,m_2}(x)\label{eq:E2f2}\\
\hbox{and} \quad (\sfE_3f)(x,y) &=& \sum_{m_3=1}^{M_3}h_{3,2,m_3}(x)h_{3,12,m_3}(x^{-1}y),\label{eq:E3f3}
\end{eqnarray}
where $M_i = |\Q_i|$ for $i=1,2,3$, and each $h_{\bullet,\bullet,\bullet}$ is a measurable function $G\to [-1,1]$.

\subsection{Estimates for structured functions}

This subsection analyzes the triple form~(\ref{eq:BMZ-corr}) when each $f_i$ is replaced by a structured approximant $\sfE_if_i$ as given by Proposition~\ref{prop:reglem}.  For these approximants, we can exert very precise control over that triple form: it turns out that it hardly depends on $g$ at all.  This fact will result from the following.

\begin{lemma}\label{lem:from-Gowers}
Let $h_1,h_2,h_{12}:G\to [-1,1]$ be measurable, and suppose that $G$ is $D$-quasirandom.  Then
\[\Big|\iint h_1(x) h_2(y) h_{12}(x^{-1}y)\,\d x\,\d y - \Big( \int h_1\Big)\Big( \int h_2\Big)\Big( \int h_{12}\Big)\Big| \leq D^{-1/2}.\]
\end{lemma}

\begin{proof}\smartqed
This is essentially the implication (v) $\Longrightarrow$ (iv) in~\cite[Theorem 4.5]{Gow08}.  Let $R$ be the left-action of $G$ on itself.  The left-hand side above may be re-written as
\begin{multline*}
\Big|\int h_1(x)\Big(\int h_2(y) h_{12}(x^{-1}y)\,\d y - \Big( \int h_2\Big)\Big( \int h_{12}\Big)\Big)\,\d x\Big|\\
= \Big|\int h_1(x)\Big(\big\langle h_2, h_{12}R^{x^{-1}}\big\rangle_{L^2(G)} - \Big( \int h_2\Big)\Big( \int h_{12}\Big)\Big)\,\d x\Big|.
\end{multline*}
By the Cauchy--Bunyakowski--Schwartz inequality, this is at most
\[\Big(\int \Big(\langle h_2, h_{12}R^{x^{-1}}\rangle_{L^2(G)} - \Big( \int h_2\Big)\Big( \int h_{12}\Big)\Big)^2\,\d x\Big)^{1/2},\]
and Lemma~\ref{lem:quasirand-mixing} bounds this by $D^{-1/2}$. \qed
\end{proof}

\begin{corollary}\label{cor:struct-fn-const}
Suppose that $G$ is $D$-quasirandom, and that
\begin{multline*}
f_1(x,y) = h_{1,1}(y)h_{1,12}(x^{-1}y), \quad f_2(x,y) = h_{2,1}(y)h_{2,2}(x)\\ \hbox{and} \quad f_3(x,y) = h_{3,2}(x)h_{3,12}(x^{-1}y)
\end{multline*}
for some measurable functions $h_{\bullet,\bullet}:G\to [-1,1]$.  Then the quantity
\[\phi(g) := \int f_1\cdot f_2S^g\cdot f_3S^gT^g\,\d\mu\]
satisfies
\[|\phi(g) - \phi(g')| \leq 2D^{-1/2} \quad \forall g,g' \in G.\]
\end{corollary}

\begin{proof}\smartqed
For this choice of functions $f_i$, one has
\[\phi(g) = \iint (h_{1,1} h_{2,1})(y)\cdot (h_{1,12} h_{3,12})(x^{-1}y)\cdot (h_{2,2} h_{3,2})(gx)\,\d x\,\d y.\]
Applying Lemma~\ref{lem:from-Gowers} to this integral, we find that it lies within $D^{-1/2}$ of
\[\Big(\int h_{1,1} h_{2,1}\Big)\Big(\int h_{1,12} h_{3,12}\Big)\Big(\int h_{2,2} h_{3,2}\Big)\]
for all $g$. \qed
\end{proof}

\begin{proposition}\label{prop:from-B-to-C}
Let $\eps > 0$, and let $\Q_i$, $i=1,2,3$, be partitions as in~(\ref{eq:Q1}),~(\ref{eq:Q2}) and~(\ref{eq:Q3}).  Assume that $G$ is $D$-quasirandom for some
\begin{eqnarray}\label{eq:lower-bound-on-D-eta}
D \geq \frac{16|\Q_1|^2|\Q_2|^2|\Q_3|^2}{\eps^2}.
\end{eqnarray}
Let the sets $E$ and $B$ be as in Theorem C.  Then $B$ contains the set
\begin{multline*}
C := \Big\{g \in G:\ \mu\big(E \cap (g^{-1},e)\cdot E\cap (g^{-1},g^{-1})\cdot E\big) \\ \geq \int (\sfE_11_E)\cdot (\sfE_21_E)S^g \cdot (\sfE_31_E)S^gT^g\,\d\mu - \eps/2\Big\}.
\end{multline*}
\end{proposition}

\begin{proof}\smartqed
According to the decompositions~(\ref{eq:E1f1})--(\ref{eq:E3f3}), the quantity
\[\psi(g) := \int (\sfE_11_E)\cdot (\sfE_21_E)S^g \cdot (\sfE_31_E)S^gT^g\,\d\mu\]
is a sum of at most $|\Q_1||\Q_2||\Q_3|$ quantities having the form $\phi$ treated by Corollary~\ref{cor:struct-fn-const}.  Therefore, by that corollary, $\psi(g)$ varies by at most
\[2|\Q_1||\Q_2||\Q_3| D^{-1/2} \leq \eps/2\]
as $g$ varies in $G$.  Therefore
\[\mu\big(E \cap (g^{-1},e)\cdot E\cap (g^{-1},g^{-1})\cdot E\big) \geq \psi(e) - \eps = \int \sfE_11_E\cdot \sfE_21_E\cdot \sfE_31_E\,\d\mu - \eps\]
for all $g \in C$.

Finally, an appeal to Lemma~\ref{lem:Chu} gives
\[\int \sfE_11_E\cdot \sfE_21_E\cdot \sfE_31_E\,\d\mu \geq \int 1_E\cdot \sfE_11_E\cdot \sfE_21_E\cdot \sfE_31_E\,\d\mu \geq \mu(E)^4.\]
Hence $g \in B$ for all $g \in C$.
\qed\end{proof}

In the remainder of the proofs of Theorems C and D, we show that the set $C$ from the above corollary is large in the required senses, rather than handle the set $B$ from Theorem C itself.

\subsection{Completed proof of Theorem C}

In order to use the approximants $\sfE_if_i$ to the functions $f_i$ in estimating the triple forms~(\ref{eq:BMZ-corr}), we need to convert a bound on the norms such as $\|\cdot\|_{\check{\otimes}_{1,2}}$ into some more direct control on those triple forms.  That control is given by the next proposition.

\begin{proposition}\label{prop:kill-random-again}
Let $f_1,f_2,f_3:X\to [-1,1]$ be measurable.  Then
\begin{multline*}
\int\Big|\int f_1\cdot f_2S^g\cdot f_3S^gT^g\,\d \mu\Big|\,\d g\\ \leq \sqrt{2D^{-1/4} + \sqrt{\min\{\|f_1\|_{\check{\otimes}_{1,12}},\|f_2\|_{\check{\otimes}_{1,2}},\|f_3\|_{\check{\otimes}_{12,2}}\}}}.
\end{multline*}
\end{proposition}

This may be viewed as the analog of Proposition~\ref{prop:kill-random-f3} for BMZ corners.  It will be deduced from the following intermediate estimates.

\begin{lemma}\label{lem:lem-for-kill-random}
For any $f:X\to [-1,1]$, the following inequalities hold:
\begin{eqnarray}
&&\int \|\sfE(f\cdot fS^h\,|\,\S_X^T)\|_2^2\,\d h \leq D^{-1/2} + \|f\|_{\check{\otimes}_{1,2}}, \label{eq:bound-by-1-2}\\
&&\int \|\sfE(f\cdot f(ST)^h\,|\,\S_X^T)\|_2^2\,\d h \leq D^{-1/2} + \|f\|_{\check{\otimes}_{12,2}}, \label{eq:bound-by-12-2}\\
\hbox{and} \quad && \int \|\sfE(f\cdot f(ST)^h\,|\,\S_X^S)\|_2^2\,\d h \leq D^{-1/2} + \|f\|_{\check{\otimes}_{1,12}}. \label{eq:bound-by-1-12}
\end{eqnarray}
\end{lemma}

\begin{proof}\smartqed
We first prove~(\ref{eq:bound-by-1-2}).  Define a new probability $(G\times G)$-space as follows: let $Y := G\times G\times G$, let $\nu := m_Y$, and let the two generating actions be
\[\t{S}^g(x,y,z) = (gx,y,z) \quad \hbox{and} \quad \t{T}^g(x,y,z) = (x,gy,gz).\]
(In ergodic-theory terms, this is the relative product of two copies of $(X,\mu,S,T)$ over $\S_X^T$.)  Then a simple calculation shows that
\[\|\sfE(f\cdot fS^h\,|\,\S_X^T)\|_2^2 = \int F\cdot F\t{S}^h\,\d\nu = \big\langle F, F\t{S}^h \big\rangle_{L^2(\nu)},\]
where
\[F(x,y,z) := f(x,y)f(x,z).\]

Let $H := \sfE(F\,|\,\S_Y^{\t{S}})$.  Applying Lemma~\ref{lem:quasirand-mixing} to $F$ and $\t{S}$, it follows that
\begin{multline*}
\int\|\sfE(f\cdot fS^h\,|\,\S_X^T)\|_2^2\,\d h = \int \big\langle F, F\t{S}^h \big\rangle_{L^2(\nu)}\,\d h\\
\leq \int \Big|\big\langle F,F\t{S}^h\big\rangle_{L^2(\nu)} - \big\langle F,H\big\rangle_{L^2(\nu)}\Big|\,\d h + \big|\big\langle F,H\big\rangle_{L^2(\nu)}\big|\\
\leq D^{-1/2} + \big|\big\langle F,H\big\rangle_{L^2(\nu)}\big|.
\end{multline*}

On the other hand, $H$ is a function of only the coordinates $y$ and $z$ for a point $(x,y,z) \in Y$, and therefore
\begin{eqnarray*}
\big|\big\langle F,H\big\rangle_{L^2(\nu)}\big| &=& \Big|\iiint f(x,y)f(x,z)H(y,z)\,\d x\,\d y\,\d z\Big|\\ &\leq& \int\Big|\iint f(x,y)\cdot (f(x,z)H(y,z))\,\d x\,\d y\Big|\,\d z\\
&\leq& \int \|f\|_{\check{\otimes}_{1,2}}\,\d z = \|f\|_{\check{\otimes}_{1,2}}.
\end{eqnarray*}

The proof of~(\ref{eq:bound-by-12-2}) is very similar.  One uses the same auxiliary $(G\times G)$-space $(Y,\nu,\t{S},\t{T})$ and function $F$ as before, but now one proceeds from the estimate
\begin{multline*}
\int\|\sfE(f\cdot f(ST)^h\,|\,\S_X^T)\|_2^2\,\d h = \iint F\cdot F(\t{S}\t{T})^h\,\d\nu\,\d h\\ = \int \big\langle F, F(\t{S}\t{T})^h \big\rangle_{L^2(\nu)}\,\d h
\leq D^{-1/2} + \Big|\int F\cdot \sfE(F\,|\,\S_Y^{\t{S}\t{T}})\,\d \nu\Big|.
\end{multline*}
The function $H' := \sfE(F\,|\,\S^{\t{S}\t{T}}_Y)$ can be written as a function of only $x^{-1}y$ and $x^{-1}z$ for a point $(x,y,z) \in Y$.  Using this, the change of variables $w:= x^{-1}z$ gives
\begin{eqnarray*}
\Big|\int F\cdot H'\,\d\nu\Big| &=& \Big|\iiint f(x,y)f(x,z)H'(x^{-1}y,x^{-1}z)\,\d x\,\d y\,\d z\Big|\\
&=&\Big|\iiint f(x,y)f(x,xw)H'(x^{-1}y,w)\,\d x\,\d y\,\d w\Big| \\ &\leq& \int\Big|\iint f(x,y)\cdot (f(x,xw)H'(x^{-1}y,w))\,\d x\,\d y\Big|\,\d w\\
&\leq& \int \|f\|_{\check{\otimes}_{12,2}}\,\d w = \|f\|_{\check{\otimes}_{12,2}}.
\end{eqnarray*}

Finally, inequality~(\ref{eq:bound-by-1-12}) is simply~(\ref{eq:bound-by-12-2}) with the r\^oles of $S$ and $T$ reversed.
\qed\end{proof}

\begin{proof}[Proposition~\ref{prop:kill-random-again}]\smartqed
We first prove the bound that uses either $\|f_2\|_{\check{\otimes}_{1,2}}$ or $\|f_3\|_{\check{\otimes}_{12,2}}$.  Let $u_g := f_2S^g\cdot f_3 S^gT^g$.  By Lemma~\ref{lem:vdC}, it suffices to prove that
\begin{multline*}
\iint |\langle u_g,u_{g'}\rangle|\,\d g\,\d g' = \iint |\langle u_g,u_{hg}\rangle|\,\d g\,\d h\\ \leq 2D^{-1/4} + \sqrt{\min\{\|f_2\|_{\check{\otimes}_{1,2}},\|f_3\|_{\check{\otimes}_{12,2}}\}}.
\end{multline*}

For any $g,h \in G$, one has
\begin{eqnarray*}
|\langle u_g,u_{hg}\rangle| &=& \Big|\int f_2S^g\cdot f_2S^{hg}\cdot f_3 S^gT^g\cdot f_3 S^{hg}T^{hg}\,\d\mu\Big|\\
&=& \Big|\int f_2\cdot f_2S^h\cdot f_3 T^g\cdot f_3S^hT^{hg}\,\d\mu\Big|\\
&=& |\langle F_h,F'_hT^g\rangle|,
\end{eqnarray*}
where $F_h := f_2\cdot f_2S^h$ and $F'_h := f_3\cdot f_3(ST)^h$, both of which take values in $[-1,1]$.  Therefore Lemma~\ref{lem:quasirand-mixing} and the Cauchy--Bunyakowski--Schwartz Inequality give
\begin{eqnarray*}
\iint |\langle u_g,u_{hg}\rangle|\,\d g\,\d h &\leq& \int\big|\big\langle \sfE(F_h\,|\,\S^T_X),\sfE(F'_h\,|\,\S^T_X)\big\rangle\big|\,\d h + D^{-1/2}\\
&\leq& \int \|\sfE(F_h\,|\,\S^T_X)\|_2\|\sfE(F'_h\,|\,\S^T_X)\|_2\,\d h + D^{-1/2}\\
&\leq& \sqrt{\int \|\sfE(F_h\,|\,\S^T_X)\|^2_2\,\d h\cdot \int \|\sfE(F'_h\,|\,\S^T_X)\|_2^2\,\d h} + D^{-1/2}.
\end{eqnarray*}
Each of the integrals inside this last square root is certainly at most $1$. Therefore, by Lemma~\ref{lem:lem-for-kill-random}, the last line above may be bounded by
\begin{multline*}
\hbox{either} \quad \sqrt{\|f_2\|_{\check{\otimes}_{1,2}} + D^{-1/2} } + D^{-1/2} \leq \sqrt{\|f_2\|_{\check{\otimes}_{1,2}}} + 2D^{-1/4}\\
 \hbox{or}\quad  \sqrt{\|f_3\|_{\check{\otimes}_{12,2}} + D^{-1/2}} + D^{-1/2} \leq \sqrt{\|f_3\|_{\check{\otimes}_{12,2}}} + 2D^{-1/4}.
\end{multline*}

Finally, the proof of the bound using $\|f_1\|_{\check{\otimes}_{1,12}}$ is exactly analogous to the case of $\|f_3\|_{\check{\otimes}_{12,2}}$ once one makes the substitution $g' := g^{-1}$ to write
\[\int\Big| \int f_1 \cdot f_2S^g\cdot f_3S^gT^g\,\d\mu\Big|\,\d g = \int\Big| \int f_3\cdot f_2T^{g'}\cdot f_1 S^{g'}T^{g'}\,\d\mu\Big|\,\d g'.\]
\qed\end{proof}

\begin{proof}[Theorem C]\smartqed
Let $f:= 1_E$ and $\eta := \eps^8/(4\cdot 6^4)$.  Then there are certainly values of $D$ satisfying~(\ref{eq:D-lower-bound}) for which
\begin{eqnarray}\label{eq:another-D-eta-bound}
\sqrt{2D^{-1/4} + \sqrt{\eta}} \leq \eps^2/6.
\end{eqnarray}

For this function $f$ and error tolerance $\eta$, let the partitions $\Q_i$, $i=1,2,3$, and corresponding operators $\sfE_i$ be given by Proposition~\ref{prop:reglem} and its two variants.  Then $|\Q_i| \leq \exp(2(1/\eta)^{\rm{O}(1)})$ for each $i$, and so there are values of $D$ satisfying~(\ref{eq:D-lower-bound}) for which~(\ref{eq:lower-bound-on-D-eta}) also holds.  Therefore Proposition~\ref{prop:from-B-to-C} applies, and so it suffices to show that the set $C$ from that proposition has measure at least $1-\eps$.  We will in fact prove that
\[\int\Big| \int f\cdot fS^g\cdot fS^gT^g\,\d\mu - \int (\sfE_1f)\cdot (\sfE_2f)S^g\cdot (\sfE_3f)S^gT^g \,\d\mu\Big|\,\d g \leq \eps^2/2.\]
The desired lower bound on $m_G(C)$ follows from this by Chebyshev's Inequality.

Let $f_i^\perp = f - \sfE_if$ for each $i$.  By the triangle inequality, the above estimate is a consequence of the following:
\begin{eqnarray*}
&&\int\Big| \int f\cdot fS^g\cdot f_3^\perp S^gT^g\,\d\mu\Big|\,\d g \leq \eps^2/6,\\
&&\int\Big| \int f\cdot f_2^\perp S^g\cdot (\sfE_3f)S^gT^g\,\d\mu\Big|\,\d g \leq \eps^2/6\\
\hbox{and}\quad &&\int\Big| \int f^\perp_1\cdot (\sfE_2f)S^g\cdot (\sfE_3f)S^gT^g\,\d\mu\Big|\,\d g \leq \eps^2/6.
\end{eqnarray*}
These are all now implied by Proposition~\ref{prop:kill-random-again}, together with~(\ref{eq:another-D-eta-bound}).
\qed\end{proof}

\subsection{Anti-neighbourhoods and syndeticity}

The proof of Theorem D will be based on having a large supply of fairly syndetic subsets of a quasirandom group ready to hand.  These subsets will be obtained from a simple construction in terms of representations.

Let $\pi:G\actson V$ be an orthogonal representation, let $u,v \in V$, and let $\eps > 0$. If $P^\pi = 0$ (that is, $\pi$ contains no copy of the identity representation), then let
\[A(\pi,u,v,\eps) := \{g \in G:\ |\langle u,\pi^g v\rangle| < \eps\},\]
and call this the $(\pi,u,v,\eps)$-\textbf{anti-neighbourhood}.  For general $\pi$, let 
\[A(\pi,u,v,\eps) := \{g \in G:\ |\langle u,\pi^g v\rangle - \langle P^\pi u,P^\pi v\rangle| < \eps\}.\]

If $G$ is $D$-quasirandom for some large $D$, and $u,v$ are unit vectors, then the corresponding anti-neighbourhoods are quite large: Lemma~\ref{lem:quasirand-mixing} and Chebyshev's Inequality imply that
\begin{eqnarray}\label{eq:anti-neigh-large}
m_G(A(\pi,u,v,\eps)) \geq 1 - \frac{D^{-1}}{\eps^2}.
\end{eqnarray}
This is very intuitive. If $\pi$ is irreducible with large dimension $d$, then the orbit points $\pi^gv$ should be fairly evenly spread around the high-dimensional unit sphere in $V$, and so most of them will be nearly orthogonal to any fixed direction $u$.

The present section shows that, if $G$ is highly quasirandom, then anti-neighbourhoods are also fairly syndetic.  Moreover, one can intersect a controlled number of anti-neighbourhoods, and that intersection is still fairly syndetic.  This is not implied solely by the largeness of those intersections, but it will follow from some simple inner-product estimates.  In the next section, the syndeticity of the set in Theorem D will be proved by showing that it contains such an intersection of anti-neighbourhoods.

Let us begin with a rough sketch of why a single anti-neighbourhood should be fairly syndetic, before carefully proving the result we need about intersections.  The key point is that, if $d$ is very large and we choose a moderately large number of group elements $h_1$, \ldots, $h_k$ independently at random, then with high probability the image-points $\pi^{h_i}u$ will all be nearly orthogonal to one another.  The intuition here is the same as that above: each of the random points $\pi^{h_i}u$ should be fairly evenly distributed around the high-dimensional unit sphere of $V$, independently of the others.  However, having obtained some $h_i$'s for which the vectors $\pi^{h_i}u$ are all nearly orthogonal, this then forces any other unit vector to be nearly orthogonal to at least some of them.  In particular, for any $g \in G$, there is an $i$ such that $\pi^{h_i}u$ and $\pi^gv$ are nearly orthogonal.  This implies that $u$ is nearly orthogonal to $\pi^{h_i^{-1}g}v$, and hence that $g$ is in $h_i\cdot A(\pi,u,v,\eps)$ for some small $\eps$.  Crucially, after fixing the right choice of $h_1$, \ldots, $h_k$, this argument works for every $g \in G$.

Such `near orthogonality' will be deduced using the following.

\begin{lemma}\label{lem:almost-orthog}
Let $V$ be a real Hilbert space and let $v$ and $u^1,\dots,u^m$ be unit vectors in $V$.  Suppose that
\[|\langle u^i,u^j\rangle| \leq 1/m^2 \quad \hbox{whenever}\ i\neq j.\]
Then
\[\sum_{i\leq m}\langle v,u^i\rangle^2 \leq 2.\]
\end{lemma}

\begin{proof}\smartqed
Let $a_i:= \langle v,u^i\rangle$ for each $i$, so one always has $|a_i| \leq 1$, and let
\[w := \sum_{i\leq m}a_iu^i.\]
The assumed inequalities give
\[\|w\|^2 = \sum_{i,j \leq m}a_ia_j\langle u^i,u^j\rangle \leq \sum_{i\leq m} a_i^2 + (m^2-m)/m^2 \leq \sum_{i \leq m}a_i^2 + 1,\]
and hence
\[\sum_{i\leq m}a_i^2 = \langle v,w\rangle \leq \|v\|\|w\| \leq \sqrt{\sum_{i\leq m}a_i^2 + 1}.\]
This implies that $\sum_{i\leq m}a_i^2 \leq 2$.
\qed
\end{proof}

In order to study intersections of anti-neighbourhoods, we will actually need the following crude corollary which concerns several Hilbert spaces simultaneously.

\begin{corollary}\label{cor:almost-orthog}
Let $V_1$, \dots, $V_k$ be real Hilbert spaces, and let $v_\ell$ and $u^1_\ell,\dots,u^m_\ell$ be unit vectors in $V_\ell$ for each $\ell$.  Suppose that
\[|\langle u^i_\ell,u^j_\ell\rangle| \leq 1/m^2 \quad \hbox{whenever}\ \ell \leq k,\ i\neq j.\]
Then there is some $i\leq m$ such that $|\langle v_\ell,u^i_\ell\rangle| \leq \sqrt{2k/m}$ for all $\ell \leq k$.
\end{corollary}

\begin{proof}\smartqed
Summing the inequalities proved in the preceding lemma gives
\[\sum_{i\leq m}\Big(\sum_{\ell\leq k}\langle v_\ell,u^i_\ell\rangle^2\Big) \leq 2k.\]
\qed\end{proof}

\begin{corollary}\label{cor:basic-syndetic}
Let $\eta > 0$ and $k\geq 1$, and set $K := \lceil 2k/\eta^2 + 1\rceil$ and $D:= K^6k + 1$. Let $\pi_\ell:G\actson V_\ell$ for $\ell =1,2,\ldots,k$ be orthogonal representations that are all $D$-quasirandom, and let $u_\ell,v_\ell \in V_\ell$ be unit vectors for each $\ell$.  Then
\[A := \bigcap_{\ell=1}^k A(\pi_\ell,u_\ell,v_\ell,\eta)\]
is $K$-syndetic.
\end{corollary}

\begin{proof}\smartqed
Let
\[A' := \bigcap_{\ell=1}^k A(\pi_\ell,u_\ell,u_\ell,1/K^2).\]

Let $h_1,\dots,h_K$ be a $K$-tuple of elements of $G$ drawn at random from the measure $m_G^{\otimes K}$. Then the estimate~(\ref{eq:anti-neigh-large}) and a first-moment bound give
\begin{multline*}
m_G^{\otimes K}\big\{h_i^{-1}h_j \in A'\ \forall i\neq j\ \hbox{in}\ \{1,2,\ldots,K\}\big\}\\ \geq 1 - \sum_{i\neq j}\sum_{\ell=1}^k m_G\big(G \big\backslash A(\pi_\ell,u_\ell,u_\ell,1/K^2)\big) \geq 1 - K^6kD^{-1} > 0.
\end{multline*}

This implies that there exists a $K$-tuple $h_1,\dots,h_K$ in $G$ witnessing the above event, hence such that
\[|\langle \pi_\ell^{h_i}u_\ell,\pi_\ell^{h_j}u_\ell\rangle| \leq 1/K^2 \quad \hbox{whenever}\ \ell \leq k,\ i\neq j.\]
Therefore, for any $g \in G$, Corollary~\ref{cor:almost-orthog} promises some $i \leq K$ for which
\[|\langle \pi_\ell^{h_i}u_\ell,\pi_\ell^g v_\ell\rangle| = |\langle u_\ell,\pi_\ell^{h_i^{-1}g}v_\ell\rangle| \leq \sqrt{2k/K} < \eta \quad \forall \ell \leq k,\]
so $g \in \{h_1,\dots,h_K\}\cdot A$.
\qed\end{proof}

\subsection{Completed proof of Theorem D}

The next step is the following rather technical proposition.

\begin{proposition}\label{prop:kill-random-f2}
Let $\eps > 0$ and $n \in \bbN$.  Set $k := \lceil 4/\eps^2\rceil$, and now set $\eta := 1/(3k)^8$.  Suppose that $G$ is $D$-quasirandom for some $D > 4k^4/\eta^4$.

For $\ell=2,3$, let $f^\ell_1$, $f^\ell_2$ and $f^\ell_3$ be three $[-1,1]$-valued functions, and suppose that
\[\|f^2_2\|_{\check{\otimes}_{1,2}} \leq \eta \quad \hbox{and} \quad \|f^3_3\|_{\check{\otimes}_{12,2}} \leq \eta.\]

Finally, let
\[C_\ell := \Big\{g \in G:\ \Big|\int f^\ell_1\cdot f^\ell_2S^g \cdot f^\ell_3S^gT^g\,\d\mu\Big| < \eps\Big\} \qquad \hbox{for}\ \ell=2,3.\]

Then there are elements $h_1,\ldots,h_k \in G$ and some auxiliary $[-1,1]$-valued functions
\[F^2_{2,i,j}, \quad F^2_{3,i,j}, \quad F^3_{2,i,j} \quad \hbox{and} \quad F^3_{3,i,j} \quad \hbox{for}\ 1 \leq i < j \leq k\]
such that the set
\[E := \bigcap_{1 \leq i < j \leq k} \big(A(T,F^2_{2,i,j},F^2_{3,i,j},\eta)\cap A(T,F^3_{2,i,j},F^3_{3,i,j},\eta)\big)\]
satisfies
\[E \subseteq \{h_1^{-1},\ldots,h_k^{-1}\}\cdot (C_2\cap C_3).\]
\end{proposition}

\begin{remark}
Letting $m := \lceil 2k^2/\eta^2 + 1\rceil$, the set $E$ above is $m$-syndetic by Corollary~\ref{cor:basic-syndetic} provided $D$ large enough, and so the above conclusion implies that the intersection $C_2\cap C_3$ is $(mk)$-syndetic.  However, this fact alone is not quite what we need for the proof of Theorem D: ultimately, that will require the syndeticity of the smaller intersection $C_1\cap C_2\cap C_3$ for some other set $C_1$.  In order to prove that, it will be important to have an explicit `witness' to the syndeticity of $C_2\cap C_3$ in the form of an intersection of anti-neighbourhoods, such as $E$ above.  This is why the above proposition is formulated as it is. \fin
\end{remark}

\begin{proof}\smartqed
We will prove that if $h_1$, \ldots, $h_k$ are chosen independently at random from the Haar measure $m_G$, then with positive probability one obtains a tuple for which the remaining objects required by the proposition also exist.

\vspace{7pt}

\emph{Step 1.}\quad For such a random choice of $h_1$, \ldots, $h_k$, each difference $h_jh_i^{-1}$ for $i\neq j$ is also a uniform random element of $G$, and so Lemma~\ref{lem:lem-for-kill-random} gives the estimates
\[\iint\cdots\int \|\sfE(f^2_2\cdot f^2_2S^{h_jh_i^{-1}}\,|\,\S_X^T)\|^2_2\,\d h_1\cdots\d h_{k-1}\,\d h_k \leq D^{-1/2} + \eta\]
and
\[\iint\cdots\int \|\sfE(f^3_3\cdot f^3_3(ST)^{h_jh_i^{-1}}\,|\,\S_X^T)\|^2_2\,\d h_1\cdots\d h_{k-1}\,\d h_k \leq D^{-1/2} + \eta\]
for each $i\neq j$.

Our assumptions imply that $D^{-1/2} < \eta$, and so the right-hand sides above are all less than $2\eta$.  Therefore, if $h_1$, \ldots, $h_k$ are chosen randomly as described above, then, by Chebyshev's Inequality and a first-moment bound, the event that
\begin{multline}\label{eq:all-small}
\|\sfE(f^2_2\cdot f^2_2S^{h_jh_i^{-1}}\,|\,\S_X^T)\|^2_2 < \sqrt{2\eta} \quad \hbox{for all}\ i\neq j\\ \hbox{and} \quad \|\sfE(f^3_3\cdot f^3_3(ST)^{h_jh_i^{-1}}\,|\,\S_X^T)\|^2_2 < \sqrt{2\eta} \quad \hbox{for all}\ i\neq j
\end{multline}
has probability at least
\[1 - 2k^2\sqrt{2\eta} = 1 - 2\sqrt{2}k^2/(3k)^4 > 0. \]

Therefore there exists a tuple $h_1$, \ldots, $h_k \in G$ for which the inequalities in~(\ref{eq:all-small}) all hold simultaneously.

\vspace{7pt}

\emph{Step 2.}\quad We now define the required auxiliary functions as follows:
\begin{eqnarray*}
F^2_{2,i,j} &:=& f^2_2S^{h_i}\cdot f^2_2S^{h_j} \quad \hbox{and} \quad F^2_{3,i,j} := f^2_3S^{h_i}T^{h_i}\cdot f^2_3S^{h_j}T^{h_j}\\
F^3_{2,i,j} &:=& f^3_2S^{h_i}\cdot f^3_2S^{h_j} \quad \hbox{and} \quad F^3_{3,i,j} := f^3_3S^{h_i}T^{h_i}\cdot f^3_3S^{h_j}T^{h_j}
\end{eqnarray*}
whenever $1 \leq i < j \leq k$.

\vspace{7pt}

\emph{Step 3.}\quad Having chosen those auxiliary functions, let $E$ be as as in the statement of the proposition. Suppose that $g \in E$.  We must show that $h_i g \in C_2\cap C_3$ for some $i\leq k$.

For each $g \in G$ and $\ell = 2,3$, define
\[u^\ell_g := f^\ell_2S^g\cdot f^\ell_3 S^gT^g.\]

For any $i < j$ in $\{1,2,\ldots,k\}$ and $\ell=2,3$, we have
\begin{eqnarray*}
|\langle u^\ell_{h_ig},u^\ell_{h_jg}\rangle| &=& \Big|\int f^\ell_2S^{h_ig}\cdot f^\ell_2S^{h_jg}\cdot f^\ell_3S^{h_ig}T^{h_ig}\cdot f^\ell_3S^{h_jg}T^{h_jg}\,\d\mu\Big| \\
&=& \Big|\int f^\ell_2S^{h_i}\cdot f^\ell_2S^{h_j}\cdot f^\ell_3S^{h_i}T^{h_ig}\cdot f^\ell_3S^{h_j}T^{h_jg}\,\d\mu\Big| \\
&=& \Big|\int F^\ell_{2,i,j}\cdot F^\ell_{3,i,j}T^g\,\d\mu\Big|.
\end{eqnarray*}
Since $g\in E$, in particular $g \in  A(T,F^\ell_{2,i,j},F^\ell_{3,i,j},\eta)$, and so the above is at most
\[\Big|\int \sfE(F^\ell_{2,i,j}\,|\,\S_X^T)\sfE(F^\ell_{3,i,j}\,|\,\S_X^T)\,\d\mu\Big| + \eta.\]

If $\ell=2$ then this expression is bounded by
\begin{multline*}
\|\sfE(F^2_{2,i,j}\,|\,\S_X^T)\|_2 + \eta = \|\sfE(f^2_2S^{h_i}\cdot f^2_2S^{h_j}\,|\,\S_X^T)\|_2 + \eta\\
= \|\sfE(f^2_2\cdot f^2_2S^{h_jh_i^{-1}}\,|\,\S_X^T)\|_2 + \eta < \sqrt[4]{2\eta} + \eta < 3\sqrt[4]{\eta} = 1/k^2,
\end{multline*}
by~(\ref{eq:all-small}).  Similarly, if $\ell=3$ then it is bounded by
\[\|\sfE(F^3_{3,i,j}\,|\,\S_X^T)\|_2 + \eta = \|\sfE(f^3_3\cdot f^3_3(ST)^{h_jh_i^{-1}}\,|\,\S_X^T)\|_2 + \eta < 1/k^2.\]

Thus, we have shown that
\[|\langle u^\ell_{h_ig},u^\ell_{h_jg}\rangle| < 1/k^2 \quad \hbox{whenever}\ i \neq j,\ \ell=2,3.\]

We may therefore apply Corollary~\ref{cor:almost-orthog} to the inner products
\[\langle f^\ell_1,u^\ell_{h_ig}\rangle, \quad i=1,2,\ldots,k,\ \ \ell=2,3\]
to conclude that, for any $g \in E$, there is at least one $i \leq k$ for which
\[|\langle f^2_1,u^2_{h_ig}\rangle| \leq \sqrt{4/k} < \eps \quad \hbox{and} \quad |\langle f^3_1,u^3_{h_ig}\rangle| < \eps.\]
For this choice of $i$ one has $h_ig \in C_2\cap C_3$, as required.
\qed\end{proof}

\begin{proof}[Theorem D]\smartqed
Let $f:= 1_E$.  As for Theorem C, this proof will be based on three different decompositions of $f$ as given by Proposition~\ref{prop:reglem}.  However, a new complication here is that the partition $\Q_1$ will need to be chosen after $\Q_2$ and $\Q_3$, and considerably finer than those others.

In the course of the proof, we will meet three points at which we require a lower bound on $D$.  All of these lower bounds will be satisfied for some $D$ as in~(\ref{eq:D-lower-bound2}), so there is a choice of $D$ of the form in~(\ref{eq:D-lower-bound2}) for which the whole proof can be carried out.

\vspace{7pt}

\emph{Step 1.}\quad Set $k := \lceil 36/\eps^2\rceil$ and $\eta := 1/(3k)^8$.  Let $\Q_2$ and $\Q_3$ be partitions as given by the two variants of Proposition~\ref{prop:reglem} for this error tolerance $\eta$, and consider the resulting decompositions
\[f = f_2^\perp + \sfE_2f = f_3^\perp + \sfE_3f.\]
Let
\[C_2 := \Big\{g \in G:\ \Big|\int f\cdot f_2^\perp S^g \cdot fS^gT^g\,\d\mu\Big| < \eps/3 \Big\}\]
and
\[C_3 := \Big\{g \in G:\ \Big|\int f\cdot (\sfE_2f)S^g \cdot f_3^\perp S^gT^g\,\d\mu\Big| < \eps/3\Big\}.\]

Given the above choice of $k$ and $\eta$, there is a $D$ as in~(\ref{eq:D-lower-bound2}) which satisfies $D > 4k^4/\eta^4$ (indeed, at this point~(\ref{eq:D-lower-bound2}) leaves vastly more room than we need). Therefore Proposition~\ref{prop:kill-random-f2} applies to give $h_1,\ldots,h_k \in G$ and some auxiliary $[-1,1]$-valued functions
\[F^2_{2,i,j}, \quad F^2_{3,i,j}, \quad F^3_{2,i,j} \quad \hbox{and} \quad F^3_{3,i,j} \quad \hbox{for}\ 1 \leq i < j \leq k\]
such that $E \subseteq \{h_1^{-1},\ldots,h_k^{-1}\}\cdot (C_2\cap C_3)$, where
\[E = \bigcap_{1 \leq i < j \leq k}\big(A(T,F^2_{2,i,j},F^2_{3,i,j},\eta) \cap A(T,F^3_{2,i,j},F^3_{3,i,j},\eta)\big).\]

Let $M_i := |\Q_i|$ for $i=2,3$.

\vspace{7pt}

\emph{Step 2.}\quad Now set $\eta' := \eps/6M_2M_3$, let $\Q_1$ be as given by Proposition~\ref{prop:reglem} for this error tolerance $\eta'$, and consider the decomposition
\[f = f_1^\perp + \sfE_1f.\]
Let
\[C_1 := \Big\{g \in G:\ \Big|\int f_1^\perp\cdot (\sfE_2f) S^g \cdot (\sfE_3f)S^gT^g\,\d\mu\Big| < \eps/3 \Big\}.\]

Observe that
\[\eta' \geq (\eps/6)\exp\big(-4(1/\eta)^{\rm{O}(1)}\big) \geq \exp\big(-(1/\eps)^{\rm{O}(1)}\big)\]
for $\eps \in (0,1/2)$.  Therefore
\[\frac{16|\Q_1|^2|\Q_2|^2|\Q_3|^2}{\eps^2} \leq \frac{16\exp\big(4(1/\eta')^{\rm{O}(1)} + 8(1/\eta)^{\rm{O}(1)}\big)}{\eps^2} \leq \exp\big(\exp\big((1/\eps)^{\rm{O}(1)}\big)\big)\]
for all $\eps \in (0,1/2)$, and so there are values of $D$ satisfying~(\ref{eq:D-lower-bound2}) for which~(\ref{eq:lower-bound-on-D-eta}) holds for these partitions $\Q_1$, $\Q_2$ and $\Q_3$.  (This is the only point at which we need the double exponential in~(\ref{eq:D-lower-bound2}).) Therefore Proposition~\ref{prop:from-B-to-C} applies, and so it suffices to show that the set $C$ from that proposition is $K$-syndetic.  Moreover, that set $C$ clearly contains $C_1 \cap C_2 \cap C_3$, so it suffices to show that this triple intersection is $K$-syndetic.

\vspace{7pt}

\emph{Step 3.}\quad We now make use of the decompositions~(\ref{eq:E2f2}) and~(\ref{eq:E3f3}).  Substituting these into the integral that appears inside the definition of $C_1$, we obtain
\begin{eqnarray}
\sum_{m_2,m_3} \iint \big(f_1^\perp(x,y)h_{2,1,m_2}(y)h_{3,12,m_3}(x^{-1}y)\big)(h_{2,2,m_2}h_{3,2,m_3})(gx)\,\d x\,\d y. \nonumber\\ \label{eq:correlation-decomposed}
\end{eqnarray}

Let
\[\psi_{m_2,m_3}(x,y) := f_1^\perp(x,y)h_{2,1,m_2}(y)h_{3,12,m_3}(x^{-1}y),\]
let
\[\phi_{m_2,m_3}(x,y) := (h_{2,2,m_2}h_{3,2,m_3})(x)\]
(so $\phi$ depends only nominally on $y$), and let
\[E' := \bigcap_{m_2,m_3}A\big(S,\psi_{m_2,m_3},\phi_{m_2,m_3},\eta'\big).\]

We will now show that $E' \subseteq C_1$, so suppose that $g \in E'$.  Then the definition of $E'$ and the expression~(\ref{eq:correlation-decomposed}) give
\begin{multline*}
\Big|\int f_1^\perp\cdot (\sfE_2f) S^g \cdot (\sfE_3f)S^gT^g\,\d\mu\Big| \leq \sum_{m_2,m_3}\Big|\int \psi_{m_2,m_3}\cdot \phi_{m_2,m_3}S^g\,\d\mu\Big|\\
\leq
\sum_{m_2,m_3} \Big(\Big|\int \psi_{m_2,m_3}\,\d \mu\Big|\Big|\int \phi_{m_2,m_3}\,\d\mu\Big| + \eta'\Big)\\ \leq \sum_{m_2,m_3}\Big|\int \psi_{m_2,m_3}\,\d \mu\Big| + \eps/ 6.
\end{multline*}
Substituting from the definition of $\psi_{m_2,m_3}$, this is
\begin{multline*}
\sum_{m_2,m_3} \Big|\iint f_1^\perp(x,y)h_{2,1,m_2}(y)h_{3,12,m_3}(x^{-1}y)\,\d x\,\d y\Big| + \eps/6\\
\leq M_2M_3\|f_1^\perp\|_{\check{\otimes}_{1,12}} + \eps/6 \leq M_2M_3\eta' + \eps/6 = \eps/3,
\end{multline*}
so $g \in C_1$.

\vspace{7pt}

\emph{Step 4.}\quad Finally, letting
\[E'' = \bigcap_{i=1}^k h_i^{-1}E' = \bigcap_{i=1}^k\bigcap_{m_2,m_3}A\big(S,\psi_{m_2,m_3},\phi_{m_2,m_3}S^{h_i},\eta'\big),\]
it follows that
\[E\cap E'' \subseteq \{h_1^{-1},\ldots,h_k^{-1}\}\cdot (C_1\cap C_2\cap C_3) \subseteq \{h_1^{-1},\ldots,h_k^{-1}\}\cdot C.\]
On the other hand,
\begin{multline*}
E\cap E'' = \bigcap_{1 \leq i < j \leq k}\big(A(T,F^2_{2,i,j},F^2_{3,i,j},\eta) \cap A(T,F^3_{2,i,j},F^3_{3,i,j},\eta)\big)\\ \cap \bigcap_{i=1}^k\bigcap_{m_2,m_3}A\big(S,\psi_{m_2,m_3},\phi_{m_2,m_3}S^{h_i},\eta'\big).
\end{multline*}
Since $\eps < 1/2$, this is an intersection of at most
\[2k^2 + kM_2M_3 \leq 2k^2 + k\exp(4(1/\eta)^{\rm{O}(1)}) \leq \exp\big((1/\eps)^{\rm{O}(1)})\]
anti-neighbourhoods for $\eps \in (0,1/2)$.  On the other hand,
\[\eta,\eta' \geq \exp(-(1/\eps)^{\rm{O}(1)}).\]

Let
\[K_0 := \big\lceil 2(2k^2 + kM_2M_3)/\min\{\eta,\eta'\}^2 + 1\big\rceil \leq \exp\big((1/\eps)^{\rm{O}(1)}\big).\]
There is a $D$ satisfying~(\ref{eq:D-lower-bound2}) for which
\[D \geq K_0^6\big(2k^2 + kM_2M_3\big) + 1.\]
Therefore Corollary~\ref{cor:basic-syndetic} can be applied to deduce that this intersection is still $K_0$-syndetic.  Therefore $C$ is $K$-syndetic for
\[K = K_0k = \exp\big((1/\eps)^{\rm{O}(1)}\big),\]
as required.
\qed\end{proof}

\begin{remark}
 The above proof uses three different decompositions of $f$, one for each the three positions in the triple form~(\ref{eq:BMZ-corr}).  However, these decompositions do not have equal status: the decomposition $f = f_1^\perp + \sfE_1f$ (corresponding to the first position) requires a much finer partition $\Q_1$, and depends on having already obtained the partitions $\Q_2$ and $\Q_3$ corresponding to the second and third positions.  This contrasts with the proof of Theorem C, where the three positions in the triple form~(\ref{eq:BMZ-corr}) have roughly equal status.

The reason for this difference can be seen in the proof of Proposition~\ref{prop:kill-random-again}.  First, we proved the bound that uses the norm $\|f_2\|_{\check{\otimes}_{1,2}}$ or $\|f_3\|_{\check{\otimes}_{12,2}}$; then, the bound that uses $\|f_1\|_{\check{\otimes}_{1,12}}$ was obtained by the same argument upon changing variables to $g' := g^{-1}$. However, this change of variables converts left-syndeticity to \emph{right}-syndeticity, so we cannot use it in the same way to analyze the set $C_1$ in the proof of Theorem D.  Instead, we must first replace $f$ with its structured approximations $\sfE_2 f$ (in the second position) and $\sfE_3 f$ (in the third position), and then use the special structure of those approximations to analyze the contribution of $f$ in the first position of the triple form, without changing variables.

This discussion also suggests why our methods fail to answer Question 1 from the Introduction (about syndeticity in the setting of Theorem B).  The current version of Theorem B involves an estimate of the triple form for na\"\i ve corners which uses an integral of the kind appearing in Corollary~\ref{cor:quasirand-mixing2}.  That integral features both $T^g$ and $S^{g^{-1}}$.  Once again, the appearance of $g^{-1}$ converts left-syndetic sets into right-syndetic sets, and so it is not clear how to obtain control on this integral on any particular anti-neighbourhood. \fin
\end{remark}

\begin{acknowledgement}
Research supported by a fellowship from the Clay Mathematics Institute.  I am grateful to Vitaly Bergelson for sharing~\cite{BerRobZorKra14} with me, to Ben Green for pointing me to the references~\cite{ErdStr75,Pyb97}, to Sean Eberhard for pointing me to the reference~\cite{Sol13}, and to Julia Wolf for suggesting some useful clarifications.
\end{acknowledgement}

\bibliographystyle{spmpsci}
\bibliography{bibfile}

\def\cprime{$'$}
\begin{thebibliography}{10}
\providecommand{\url}[1]{{#1}}
\providecommand{\urlprefix}{URL }
\expandafter\ifx\csname urlstyle\endcsname\relax
  \providecommand{\doi}[1]{DOI~\discretionary{}{}{}#1}\else
  \providecommand{\doi}{DOI~\discretionary{}{}{}\begingroup
  \urlstyle{rm}\Url}\fi

\bibitem{AjtSze74}
Ajtai, M., Szemer{\'e}di, E.: Sets of lattice points that form no squares.
\newblock Stud. Sci. Math. Hungar. \textbf{9}, 9--11 (1975) (1974)

\bibitem{Aus--quantBerTao--erratum}
Austin, T.: Quantitative {E}quidistribution for {C}ertain {Q}uadruples in
  {Q}uasi-{R}andom {G}roups: {E}rratum.
\newblock To appear, Combin. Probab. Comput.

\bibitem{Aus--quantBerTao}
Austin, T.: Quantitative {E}quidistribution for {C}ertain {Q}uadruples in
  {Q}uasi-{R}andom {G}roups.
\newblock Combin. Probab. Comput. \textbf{24}(2), 376--381 (2015).
\newblock \doi{10.1017/S0963548314000492}.
\newblock \urlprefix\url{http://dx.doi.org/10.1017/S0963548314000492}

\bibitem{BerMcC07}
Bergelson, V., McCutcheon, R.: Central sets and a non-commutative {R}oth
  theorem.
\newblock Amer. J. Math. \textbf{129}(5), 1251--1275 (2007).
\newblock \doi{10.1353/ajm.2007.0031}.
\newblock \urlprefix\url{http://dx.doi.org/10.1353/ajm.2007.0031}

\bibitem{BerMcCZha97}
Bergelson, V., McCutcheon, R., Zhang, Q.: A {R}oth theorem for amenable groups.
\newblock Amer. J. Math. \textbf{119}(6), 1173--1211 (1997)

\bibitem{BerRobZorKra14}
Bergelson, V., Robertson, D., Zorin-Kranich, P.: Triangles in {C}artesian
  squares of quasirandom groups.
\newblock Preprint, available online at \verb|arXiv.org|: 1410.5385

\bibitem{BerTao14}
Bergelson, V., Tao, T.: Multiple recurrence in quasirandom groups.
\newblock Geom. Funct. Anal. \textbf{24}(1), 1--48 (2014).
\newblock \doi{10.1007/s00039-014-0252-0}.
\newblock \urlprefix\url{http://dx.doi.org/10.1007/s00039-014-0252-0}

\bibitem{Brotom85}
Br\"ocker, T., tom Dieck, T.: Representations of Compact Lie Groups.
\newblock Springer (1985)

\bibitem{Chu09}
Chu, Q.: Convergence of weighted polynomial multiple ergodic averages.
\newblock Proc. Amer. Math. Soc. \textbf{137}, 1363--1369 (2009)

\bibitem{ChuZorKra14}
Chu, Q., Zorin-Kranich, P.: Lower bound in the {R}oth {T}heorem for amenable
  groups.
\newblock To appear, Ergodic Theory Dynam. Systems

\bibitem{ErdStr75}
Erd{\H{o}}s, P., Straus, E.G.: How abelian is a finite group?
\newblock Linear and Multilinear Algebra \textbf{3}(4), 307--312 (1975/76)

\bibitem{FriKan99}
Frieze, A., Kannan, R.: Quick approximation to matrices and applications.
\newblock Combinatorica \textbf{19}(2), 175--220 (1999).
\newblock \doi{10.1007/s004930050052}.
\newblock \urlprefix\url{http://dx.doi.org/10.1007/s004930050052}

\bibitem{Gow08}
Gowers, W.T.: Quasirandom groups.
\newblock Combin. Probab. Comput. \textbf{17}(3), 363--387 (2008).
\newblock \doi{10.1017/S0963548307008826}.
\newblock \urlprefix\url{http://dx.doi.org/10.1017/S0963548307008826}

\bibitem{Gow10}
Gowers, W.T.: Decompositions, approximate structure, transference, and the
  {H}ahn-{B}anach theorem.
\newblock Bull. Lond. Math. Soc. \textbf{42}(4), 573--606 (2010).
\newblock \doi{10.1112/blms/bdq018}.
\newblock \urlprefix\url{http://dx.doi.org/10.1112/blms/bdq018}

\bibitem{Pyb97}
Pyber, L.: How abelian is a finite group?
\newblock In: The mathematics of {P}aul {E}rd{\H o}s, {I}, \emph{Algorithms
  Combin.}, vol.~13, pp. 372--384. Springer, Berlin (1997).
\newblock \doi{10.1007/978-3-642-60408-9\underline{ }27}.
\newblock \urlprefix\url{http://dx.doi.org/10.1007/978-3-642-60408-9\underline{
  }27}

\bibitem{Rya02}
Ryan, R.A.: Introduction to tensor products of {B}anach spaces.
\newblock Springer Monographs in Mathematics. Springer-Verlag London, Ltd.,
  London (2002).
\newblock \doi{10.1007/978-1-4471-3903-4}.
\newblock \urlprefix\url{http://dx.doi.org/10.1007/978-1-4471-3903-4}

\bibitem{Scha50}
Schatten, R.: A {T}heory of {C}ross-{S}paces.
\newblock Annals of Mathematics Studies, no. 26. Princeton University Press,
  Princeton, N. J. (1950)

\bibitem{Shk05}
Shkredov, I.D.: On a problem of {G}owers.
\newblock Dokl. Akad. Nauk \textbf{400}(2), 169--172 (2005).
\newblock (Russian)

\bibitem{Shk06}
Shkredov, I.D.: On a problem of {G}owers.
\newblock Izv. Ross. Akad. Nauk Ser. Mat. \textbf{70}(2), 179--221 (2006).
\newblock (Russian)

\bibitem{Sol13}
Solymosi, J.: Roth-type theorems in finite groups.
\newblock European J. Combin. \textbf{34}(8), 1454--1458 (2013).
\newblock \doi{10.1016/j.ejc.2013.05.027}.
\newblock \urlprefix\url{http://dx.doi.org/10.1016/j.ejc.2013.05.027}

\bibitem{TaoVu06}
Tao, T., Vu, V.: Additive combinatorics.
\newblock Cambridge University Press, Cambridge (2006)

\end{thebibliography}

\end{document}